\documentclass{amsart}

\usepackage[centertags]{amsmath}
\usepackage{amssymb,amsthm}
\usepackage[english]{babel}
\usepackage[all]{xy}
\usepackage[active]{srcltx}
\usepackage{hyperref}
\usepackage{graphicx}

\title{Obstructions to Stably Fibering Manifolds}
\author{Wolfgang Steimle}
\address{Universit\"at Bonn\\
               Mathematisches Institut\\
               Endenicher Allee~60,
               D-53115 Bonn, Germany}
\email{steimle@math.uni-bonn.de}
\date{May 2011}


\DeclareMathOperator*{\colim}{colim}
\DeclareMathOperator*{\hocolim}{hocolim}
\DeclareMathOperator*{\holim}{holim}

\begin{document}

\theoremstyle{plain}
\newtheorem{thm}{Theorem}
\newtheorem{cor}[thm]{Corollary}
\newtheorem{lem}[thm]{Lemma}
\newtheorem{prop}[thm]{Proposition}
\newtheorem{conj}[thm]{Conjecture}
\newtheorem*{add}{Addendum}

\theoremstyle{definition}

\newtheorem{defn}[thm]{Definition}
\newtheorem{rem}[thm]{Remark}
\newtheorem{constr}[thm]{Construction}
\newtheorem{obs}[thm]{Observation}
\newtheorem{exmpl}[thm]{Example}

\theoremstyle{remark}

\newtheorem*{notat}{Notation}
\newtheorem*{smallrem}{}

%-------------------------------------------------

\newcommand{\lto}{\longrightarrow}
\newcommand{\mor}{\opn{mor}}
\newcommand{\pathto}{\rightsquigarrow}
\newcommand{\opn}{\operatorname}
\newcommand{\id}{\opn{id}}
\newcommand{\Proj}{\opn{Proj}}
\newcommand{\Cyl}{\opn{Cyl}}
\newcommand{\Sing}{\opn{Sing}}
\newcommand{\cone}{\opn{cone}}
\newcommand{\cyl}{\opn{cyl}}
\newcommand{\Wh}{\mathrm{Wh}}
\newcommand{\Whp}[1]{\Wh(\pi_1 #1)}
\newcommand{\WhPL}{\Wh^{PL}}
\newcommand{\G}{\mathrm{G}}
\newcommand{\PL}{\mathrm{PL}}
\newcommand{\TOP}{\mathrm{TOP}}
\newcommand{\DIFF}{\mathrm{DIFF}}
\newcommand{\Homeo}{\mathrm{Homeo}}
\newcommand{\hofib}{\mathrm{hofib}}
\newcommand{\FIB}{\opn{FIB}}
\newcommand{\Aut}{\opn{Aut}}
\newcommand{\Out}{\opn{Out}}
\newcommand{\map}{\opn{map}}
\newcommand{\Hom}{\opn{Hom}}
\newcommand{\End}{\opn{End}}
\newcommand{\GL}{\opn{GL}}
\newcommand{\im}{\opn{im}}
\newcommand{\coker}{\opn{coker}}
\newcommand{\simp}{\opn{simp}}
\newcommand{\res}{\opn{res}}
\newcommand{\ind}{\opn{ind}}
\newcommand{\eval}{\opn{eval}}
\newcommand{\xican}{\xi_{\opn{can}}}
\newcommand{\Wall}{\mathrm{Wall}}
\newcommand{\sr}{\opn{sr}}
\newcommand{\Ho}{\opn{Ho}}

\newcommand{\arincl}{\ar@{^{(}->}}
\newcommand{\arinclinv}{\ar@{_{(}->}}
\newcommand{\norm}[1]{\Vert #1\Vert}
\newcommand{\eqclass}[1]{#1/{\sim}}
\newcommand{\sections}[2]{\Gamma\biggl(\begin{array}{c} {#1}\\ \downarrow\\ {#2}\end{array} \biggr)}
\newcommand{\verticalhofib}[3]{\hofib_{#1}\biggl(\begin{array}{c} {#2}\\ \downarrow\\ {#3}\end{array} \biggr)}

\newcommand{\lift}[4]{\opn{Lift}\left(\renewcommand{\arraystretch}{0.5}\begin{matrix} && {#1} \\ && \downarrow \\ {#2} & \overset{#3}{\longrightarrow} & {#4} \end{matrix}\right)\renewcommand{\arraystretch}{1.0}} 
\newcommand{\smalllift}[4]{\opn{Lift}\left(\begin{smallmatrix} && {#1} \\ && \downarrow \\ {#2} & \overset{#3}{\longrightarrow} & {#4} \end{smallmatrix}\right)}

\newcommand{\Ab}{\mathrm{Ab}}
\newcommand{\RMod}{R\mathrm{-mod}}
\newcommand{\Sp}{\mathrm{Spectra}}
\newcommand{\OSp}{\Omega\mathrm{-Spectra}}
\newcommand{\Cat}{\mathrm{Cat}}
\newcommand{\op}{^{op}}
\newcommand{\ltwotop}{\ell^2\textendash \mathrm{Top}}
\newcommand{\cpCW}{\mathbf{cpCW}}
\newcommand{\cat}{\mathbf{cat}}
\newcommand{\CGHaus}{\mathbf{CGHaus}}
\newcommand{\sk}{\opn{sk}}

\newcommand{\RR}{\mathbb{R}}
\newcommand{\ZZ}{\mathbb{Z}}
\newcommand{\NN}{\mathbb{N}}
\newcommand{\QQ}{\mathbb{Q}}
\newcommand{\CC}{\mathbb{C}}
\newcommand{\calB}{\mathcal{B}}
\newcommand{\calC}{\mathcal{C}}
\newcommand{\calE}{\mathcal{E}}
\newcommand{\calH}{\mathcal{H}}
\newcommand{\calL}{\mathcal{L}}
\newcommand{\calP}{\mathcal{P}}
\newcommand{\calQ}{\mathcal{Q}}
\newcommand{\calS}{\mathcal{S}}
\newcommand{\calR}{\mathcal{R}}
\newcommand{\calU}{\mathcal{U}}
\newcommand{\Bun}{\mathrm{Bun}}
\newcommand{\Fib}{\mathrm{Fib}}
\newcommand{\fib}{\opn{fib}}
\newcommand{\Rfd}{\mathcal{R}^{fd}}
\newcommand{\Rf}{\mathcal{R}^f}
\newcommand{\Rfh}{\mathcal{R}^f_h}

%---------------------------------------------------------

\bibliographystyle{alpha}
\setcounter{secnumdepth}{1}
\numberwithin{thm}{section}

\SelectTips{eu}{10}
\renewcommand{\theenumi}{\roman{enumi}}
\renewcommand{\labelenumi}{\textup{(}\theenumi\textup{)}}

\renewcommand{\theequation}{\arabic{equation}}
\renewcommand{\figurename}{Fig.}

%---------------------------------------------------

\begin{abstract}
Is a given map between compact topological manifolds homotopic to the projection map of a fiber bundle? In this paper obstructions to this question are introduced with values in higher algebraic $K$-theory. Their vanishing implies that the given map fibers \emph{stably}. The methods also provide results for the corresponding uniqueness question; moreover they apply to the fibering of Hilbert cube manifolds, generalizing results by Chapman-Ferry.
\end{abstract}

\maketitle

\setcounter{tocdepth}{1}
\tableofcontents

\section{Introduction}

Given a map $f\colon M\to B$ between closed manifolds, is $f$ homotopic to the projection map of a fiber bundle of closed manifolds? Can the different ways of fibering $f$ be classified? These questions have a long tradition in geometric topology. In the research on high-dimensional manifolds, the investigation of these questions has accompanied the development of the subject since its beginnings: The fibering theorem of Browder-Levine \cite{Browder-Levine(1966)} was an early application of surgery techniques and the $h$-cobordism theorem. Further results have been obtained by Farrell \cite{Farrell(1971)} and Siebenmann \cite{Siebenmann(1970a)} for $B=S^1$, using the $s$-cobordism theorem and computations of the Whitehead group of semi-direct products $G\rtimes_\alpha \ZZ$. 

Casson \cite{Casson(1967)} pioneered the study of fibering questions for higher-dimensional base manifolds by considering $B=S^n$, applying techniques of surgery theory. Quinn \cite{Quinn(1970)} was the first to systematically describe block structure spaces using the $L$-theoretic assembly map and to develop a general obstruction theory to ``block fibering'' a given map. 

In the $Q$-manifold world, Chapman-Ferry \cite{Chapman-Ferry(1978)} obtained the most general results available so far. Most recently, in the finite-dimensional case, joint work of the author with Farrell and L\"uck \cite{Farrell-Lueck-Steimle(2009)} shows how the obstructions defined by Farrell and Siebenmann over $S^1$ can be generalized to arbitrary base spaces (where, however, they stop being a complete set of obstructions). 

In the light of the development of parametrized $h$-cobordism theory since the 1970s, this work re-focuses on the role of algebraic $K$-theory in fibering questions. As we will see, higher algebraic $K$-theory of spaces provides obstructions for both questions  of existence and uniqueness. Moreover, the vanishing of these obstructions has a concrete geometric meaning: The obstructions constructed in this work form a complete set of obstructions to fibering manifolds \emph{stably}. Here stabilization refers to crossing the total space with disks of sufficiently high dimension, thus leaving the category of closed manifolds. In fact, the theory of stably fibering manifolds is best formulated and proved entirely in the world of compact manifolds with boundary (which we call compact manifolds for short).

More concretely, let $f\colon M\to B$ be a map between compact topological manifolds. Then, by definition, $f$ \emph{stably fibers} if, for some $n\in\NN$, the composite
\[f\circ\Proj\colon M\times D^n\to M\to B\]
is homotopic to the projection map of a fiber bundle whose fibers are compact topological manifolds. The following questions will be dealt with: 
\begin{itemize}
\item When does $f$ stably fiber?
\item How many different ways are there for $f$ to stably fiber? Denote by $C$ the set of all bundle maps $g\colon M\times D^n\to B$ for some $n$ which are homotopic to $f\circ\Proj$. We define two elements to be \emph{equivalent}, and write $g\sim g'$, if after further stabilizing there is a bundle homeomorphism $i\colon M\times D^N\to M\times D^N$ from $g$ to $g'$ (i.e.~$i\circ g=g'$), such that $i$ is homotopic to the identity map. The precise question is then: How can $C/{\sim}$ be described?
\end{itemize}

Factor $f$ into a homotopy equivalence $\lambda$ followed by a fibration $p$. Under a finiteness assumption on the fiber $F$ of $p$, two obstructions will be defined:
\begin{itemize}
\item $\Wall(p)\in H^0(B;\Wh(F))$, which is an obstruction to reducing $p$ to a fiber bundle of compact manifolds. Here the term $\Wh$ is used to denote the (connective topological) Whitehead spectrum as defined by Waldhausen. It is defined in terms of algebraic $K$-theory of spaces and is closely connected to the classification of parametrized $h$-cobordisms. The term $H^0(B;\Wh(F_b))$ denotes a specific generalized cohomology group of $B$ with respect to the Whitehead spectrum of the fibers, where the coefficients are twisted according to the data of the fibration $p$. 
\item If $\Wall(p)$ vanishes, then there is a second obstruction $o(f)$ lying in the cokernel of a specific map
\[\pi_0(\beta)\colon H^0(B;\Omega\Wh(F))\to \Wh(\pi_1 M).\]
\end{itemize}
See section \ref{section:definition} for a precise explanation of terms. 

\begin{thm}[Existence]\label{main_result_1}
The map $f$ stably fibers if and only if the fibers of $p$ are finitely dominated, and $\Wall(p)$ and $o(f)$ both vanish.
\end{thm}

\begin{thm}[Classification]\label{main_result_2}
If $f$ stably fibers, then the set $C/{\sim}$ is in bijection with the kernel of the map $\pi_0(\beta)$.
\end{thm}

In a sense, this paper is a companion paper to \cite{Steimle(2011)} since Theorems \ref{main_result_1} and \ref{main_result_2} are rather formal consequences of the results from that paper in combination with the ``Riemann-Roch theorem with converse'' by Dwyer-Weiss-Williams \cite{Dwyer-Weiss-Williams(2003)}. Apart from defining the obstructions (in section \ref{section:definition}) and proving Theorems \ref{main_result_1} and \ref{main_result_2} (in section \ref{section:proof_of_main_results}), we provide several examples, emphasizing the ``change of total space problem'' where the more complicated Wall obstruction does not play a role. In section \ref{fqm:section_examples_II} we show that the results of \cite{Chapman-Ferry(1978)} on fibering compact $Q$-manifolds over compact ANRs can be viewed as special cases of the results presented here. The content of section \ref{fqm:section_comparison_to_FLS} is to compare the obstructions defined here with those of \cite{Farrell-Lueck-Steimle(2009)}. An appendix collects the results needed to relate the stable fibering problem presented here with the $Q$-manifold fibering problem.

\subsection*{Acknowledgement}

This work is part of my PhD thesis, written at the University of M\"unster. I thank my advisor Wolfgang L\"ueck for his constant encouragement and support and Bruce Williams for drawing my interest to the stable fibering problem and sharing his ideas. Moreover I am grateful to Arthur Bartels, Diarmuid Crowley, Bruce Hughes, Matthias Kreck and Tibor Macko for many discussions and suggestions.

%----------------------------------------------------
%----------------------------------------------------

\section{Definition of the obstructions}\label{section:definition}

Throughout this section, let $f\colon M\to B$ be a map between compact topological manifolds, and let $f=p\circ\lambda$ be a factorization into a homotopy equivalence followed by a fibration $p\colon E\to B$.

A functor from spaces to spaces is called homotopy invariant if it sends homotopy equivalences to homotopy equivalences. Given such a functor $Z$ and a fibration $p\colon E\to B$, with fiber $F$, Dwyer-Weiss-Williams \cite{Dwyer-Weiss-Williams(2003)} define a fibration $Z_B(E)\to B$, with fiber $Z(F)$, essentially by applying $Z$ ``fiber-wise'' to $p$. As Waldhausen's functor $A(X)$ is homotopy invariant, this construction leads to a fibration $A_B(E)\to B$.

Suppose that the fiber $F$ of $p$ is finitely dominated (see below). In this situation the ``parametrized $A$-theory characteristic'' \cite{Dwyer-Weiss-Williams(2003)} defines a section of the fibration $A_B(E)\to B$, up to homotopy:
\[\chi(p)\in\sections{A_B(E)}{B}\]
The natural transformation from $A(X)$ to the connective topological Whitehead spectrum $\Wh(X)$ induces a map
\[\sections{A_B(E)}{B}\to\sections{\Wh_B(E)}{B}.\]

\begin{defn}
The \emph{parametrized Wall obstruction} 
\[\Wall(p)\in\pi_0\sections{\Wh_B(E)}{B}=: H^0(B;\Wh(F))\]
of the fibration $p$ is the image of the parametrized $A$-theory characteristic under this map. 
\end{defn}

The parametrized Wall obstruction only depends of the fiber homotopy type of $p$ in the following sense: If $\varphi\colon p\to p'$ is a fiber homotopy equivalence between fibrations with fibers $F$ and $F'$ respectively, then the induced isomorphism
\[\varphi_*\colon H^0(B;\Wh(F))\to H^0(B;\Wh(F'))\]
sends $\Wall(p)$ to $\Wall(p')$. In this sense, $\Wall(p)$ only depends on $f$ rather than on the choice of factorization $f=p\circ\lambda$.

It follows from the ``Riemann-Roch theorem with converse'' \cite{Dwyer-Weiss-Williams(2003)}:

\begin{thm}\label{fqm:RRwc}
The parametrized Wall obstruction is zero if and only if there is a factorization $f=p\circ \lambda$ where $\lambda$ is a homotopy equivalence and $p$ is a fiber bundle, with fibers compact manifolds.
\end{thm}

Suppose now that we are given such a factorization. Suppose for simplicity that $B$ is connected and consider the composite
\[\beta\colon \sections{\Omega\Wh_B(E)}{B}\to\Omega\Wh(F)\xrightarrow{\chi(B)\cdot i_*} \Omega\Wh(E)\]
where the first map is the restriction map onto a chosen base point of $b$ and the second map is induced by the inclusion $F:=p^{-1}(b)\to E$ followed by multiplication with the Euler characteristic $\chi(B)\in\ZZ$.

\begin{defn}
The fibering obstruction $o(f)$ is the class of the Whitehead torsion $\tau(\lambda)$ in the cokernel of
\[\pi_0(\beta)\colon H^0(B;\Omega\Wh(F))\to\pi_0 \Omega\Wh(E)\cong \Whp{E}.\]
\end{defn}

\begin{rem}
\begin{enumerate}
\item Recall that a space $X$ is called finitely dominated if there is a finite CW complex $Y$ together with maps $i\colon X\to Y$ and $r\colon Y\to X$ such that $r\circ i\simeq\id_X$.
\item If $X$ is not connected, then the group $\Whp{X}$ should be read as the direct sum of the Whitehead groups of $\pi_1(C)$ for all path components $C$ of $X$.
\item If $B$ is not connected, the map $\beta$ is defined as the sum of the corresponding maps for the individual components.
\item See section \ref{section:proof_of_main_results} for a proof that the fibering obstruction does not depend on the choice of factorization $f=p\circ\lambda$.
\end{enumerate}
\end{rem}

We finish this section by a spectral sequence analysis of the parametrized Wall obstruction.

\begin{thm}\label{fqm:spectral_sequence}
\begin{enumerate}
\item Let $E\to B$ be a fibration over a CW complex, with fiber $F_b$ over $b$. There is a 4th quadrant spectral sequence
\[E_2^{p,q} = H^p(B;\pi_{-q}\Wh(F_b)) \Longrightarrow H^{p+q}(B;\Wh(F_b)),\]
where the $E_2$-term consists of ordinary cohomology with twisted coefficients in the system of abelian groups $\{b\mapsto\pi_{-q}\Wh(F_b)\}$.
\item If $B$ is $d$-dimensional, $d<\infty$, then the corresponding filtration 
\[\dots\supset\mathcal{F}^{p,q}\supset\mathcal{F}^{p+1,q-1}\supset\dots\]
of $H^{p+q}(B;\Wh(F_b))$ is finite, and the spectral sequence converges in the strongest possible sense, i.e.~we have
\begin{align*}
\mathcal{F}^{0,n} & =H^n(B;\Wh(F_b)) &\mathrm{for~all~}n\\
\mathcal{F}^{d+1,n-d-1} & =0 & \mathrm{for~all~}n\\
\mathcal{F}^{p,q}/\mathcal{F}^{p+1,q-1} & \cong E_\infty^{p,q} &\mathrm{for~all~}p,q
\end{align*}
\item Under the edge homomorphism 
\[e\colon H^0(B;\Wh(F_b))\to H^0(B;\pi_0\Wh(F_b))\subset \prod_{[b]\in\pi_0 B} \tilde K_0(\ZZ[\pi_1 F_b]),\]
the image of $\Wall(p)$ is the finiteness obstruction of the fiber. 
\item Suppose that all the fibers are homotopy equivalent to finite CW complexes, so that $e(\Wall(p))=0$. Let $\gamma\colon S^1\to B$ be a loop. The naturally defined secondary homomorphism
\[\ker(e)\to H^1(B;\pi_1\Wh(F_b)),\]
followed by the restriction map 
\[\gamma^*\colon H^1(B;\pi_1\Wh(F_b))\to H^1(S^1;\pi_1\Wh(F_b))\cong\Whp{F_b}_{\pi_1(S^1)}\]
(coinvariants under the $\pi_1(S^1)$-action) sends $\Wall(p)$ to the element defined by the Whitehead torsion of the fiber transport $t_\gamma$ along $\gamma$.
\end{enumerate}
\end{thm}

\begin{rem}
In the situation of (iv), the Whitehead torsion $\tau(t_\gamma)\in\Wh(\pi_1 F_b)$ is not well-defined, since $F_b$ comes with no CW structure. However, after choosing a homotopy equivalence $h\colon X\to F_b$ from some CW complex $X$, one may consider the Whitehead torsion 
\[h_*\tau(h^{-1}\circ t_\gamma\circ h)\in\Wh(\pi_1 F_b);\]
it is not hard to see that its class in $\Wh(\pi_1 F_b)_{\pi_1(S^1)}$ is independent of the choice of $X$ and $h$.
\end{rem}

\begin{proof}[Proof of Theorem \ref{fqm:spectral_sequence}]
For part (i) and (ii), assume that $B$ is the geometric realization of a simplicial set $B_\bullet$. The rule which assigns to a simplex $\sigma$ of $B_\bullet$ the pull-back $E_\sigma:=\vert\sigma\vert^*E$ defines a functor on the simplex category $\simp B_\bullet$. There is a weak homotopy equivalence \cite{Dwyer-Weiss-Williams(2003)}
\[\sections{\Wh_B(E)}{B}\simeq \holim_{\sigma\in\simp B_\bullet} \Wh(E_\sigma).\]
The spectral sequence in question is the Bousfield-Kan spectral sequence of the right-hand side. Part (iii) and (iv) follow from a close examination of the homomorphisms in question and the identification of the higher Whitehead torsion with the classical one in the unparametrized setting \cite{Steimle(2011)}. For more details, consult \cite{Steimle(2010)}.
\end{proof}

\section{Proof of Theorems \ref{main_result_1} and \ref{main_result_2}}\label{section:proof_of_main_results}

Given a fibration $p\colon E\to B$, the \emph{structure space} $\calS_n(p)$ is defined as the geometric realization of the simplicial set $\calS_n(p)_\bullet$, where a $k$-simplex is given by a commutative diagram
\begin{equation}\label{diag:structure_on_fibration}\xymatrix{
E' \ar[rd]_q \ar[rr]^\lambda_\simeq && E\times \Delta^k \ar[ld]^{p\times\id_{\Delta^k}} \\
& B\times \Delta^k
}\end{equation}
in which $q$ is a bundle of compact topological manifolds and $\lambda$ is a (fiber) homotopy equivalence. The simplicial operations are induced by pull-back. (Strictly speaking we always have to assume that $E'$ is a subset of $B\times\calU$ for a chosen ``universe'' $\calU$. See \cite{Steimle(2011)} for more details.)

If $B$ is a point, then we write $\calS_n(E)$ for $\calS_n(p)$. In the case where $B$ is a compact topological $k$-manifold, the \emph{geometric assembly map}
\[\alpha\colon\calS_n(p)\to\calS_{n+k}(E)\]
is essentially given by ``forgetting $B$''. More precisely it sends a simplex $(q,\lambda)$ as in the diagram \eqref{diag:structure_on_fibration} to the simplex $(q',\lambda)$ where $q'\colon E'\to\Delta^k$ is the composite of $q$ with the projection, and $\lambda$ is now considered as a fiber homotopy equivalence over $\Delta^k$ only.

Here is the key observation that connects the fibering question with the geometric assembly map.

\begin{lem}\label{fqm:question_formulated_with_structure_spaces}
let $f\colon M^{n+k}\to B^k$ be a map between compact topological manifolds, and let $f=p\circ \lambda$ be a factorization of $f$ into a homotopy equivalence, followed by a fibration $p\colon E\to B$.
\begin{enumerate}
\item The fibration $p$ is fiber homotopy equivalent to a bundle of compact topological $n$-manifolds if and only if $\calS_n(p)$ is non-empty.
\item $f$ is homotopic to a bundle of $n$-manifolds if and only if the element defined by $\lambda\colon M\to E$ is in the image of the map $\pi_0(\alpha)\colon\pi_0\calS_n(p)\to\pi_0\calS_{n+k}(E)$.
\item If $g,g'\colon M\to B$ are two fiber bundle projections homotopic to $f$, say that they are equivalent if there is a commutative diagram
\[\xymatrix{
M \ar[rr]^i_\cong \ar[rd]_g && M \ar[ld]^{g'}\\
& B
}\]
where $i$ is a homeomorphism which is homotopic to the identity. Then, the equivalence classes of fiber bundle projections homotopic to $f$ are in bijection to the preimage of $[\lambda]$ under the map $\alpha\colon\pi_0\calS_n(p)\to\pi_0\calS_{n+k}(E)$.
\end{enumerate}
\end{lem}

\begin{proof}
(i) is true by definition, and (ii) follows from (iii). Statement (iii) is basically a close examination of the definition. 

Indeed, as $\calS_n(p)_\bullet$ is Kan, an element in the preimage of $[\lambda]$ under $\pi_0(\alpha)$ is given by a commutative diagram
\begin{equation}\label{fqm:representing_diagram1}
\xymatrix{
N \ar[rd]_q \ar[rr]^\varphi_\simeq && E \ar[ld]^p \\
& B
}\end{equation}
with $q$ a bundle of compact $n$-manifolds, such that $N$ is homeomorphic to $M$ via a map under which $\varphi$ corresponds to $\lambda$ up to homotopy. It defines the same element as the diagram
\[\xymatrix{
N' \ar[rd]_{q'} \ar[rr]^{\varphi'}_\simeq && E \ar[ld]^p \\
& B
}\]
if and only if both elements form the boundaries of a similar diagram over $B\times I$. This means that both diagrams extend to a diagram
\[\xymatrix{
N' \ar[rrd]_{q'} \ar[rr]^i_{\cong} \ar@/^2pc/[rrrr]^{\varphi'} && N \ar[d]^q \ar[rr]^{\varphi} && E\ar[lld]^p\\
&&B
}\]
with $i$ a homeomorphism of bundles over $B$, such that the lower triangles commute strictly and the upper triangle commutes up to a homotopy over $B$.

Suppose that $f$ is homotopic to a bundle $g$ of $n$-manifolds. Then a choice of homotopy from $f$ to $g$ induces a fiber homotopy equivalence $\varphi\colon M\to E$ from $g$ to $p$ together with a homotopy from $\varphi$ to $\lambda$. Setting in the diagram \eqref{fqm:representing_diagram1} $N:=M$ and $q:=g$ we obtain a corresponding element in the preimage of $[\lambda]$ under $\alpha$. It is not hard to see that this rule induces a bijection between equivalence classes of fiber bundle projections homotopic to $f$ and the preimage of $[\lambda]$.
\end{proof}

Denote by $I$ the unit interval and by $p\times I$ the obvious fibration $E\times I\to B$. The stabilization map
\[\sigma\colon\calS_n(p)\to\calS_{n+1}(p\times I)\]
sends $(q,\lambda)$ to $(q\times I,\lambda\times\id_I)$; let
\[\calS_\infty(p):=\hocolim_n \calS_n(p\times I^n).\]

Clearly the geometric assembly map extends to a stable version
\[\alpha\colon \calS_\infty(p)\to\calS_\infty(E).\]

Here is a stabilized version of Lemma \ref{fqm:stable_question_formulated_with_structure_spaces}. It follows from Lemma \ref{fqm:stable_question_formulated_with_structure_spaces} together with the fact that 
\[\colim_n \pi_0 \calS_n(p)\xrightarrow{\cong}\pi_0\hocolim_n \calS_n(p).\]

\begin{lem}\label{fqm:stable_question_formulated_with_structure_spaces}
\begin{enumerate}
\item The fibration $p$ is fiber homotopy equivalent to a bundle of compact topological manifolds if and only if $\calS_\infty(p)$ is non empty.
\item A map $f\colon M\to B$ stably fibers if and only if the element defined by $\lambda\colon M\to E$ is in the image of the map $\pi_0(\alpha)\colon\pi_0\calS_\infty(p)\to\pi_0\calS_\infty(E)$.
\item Recall the set $C/{\sim}$ from the introduction. There is a bijection from $C/{\sim}$ to the preimage of $[\lambda]$ under the map $\alpha\colon\pi_0\calS_\infty(p)\to\pi_0\calS_\infty(E)$.
\end{enumerate}
\end{lem}

The main result of \cite{Steimle(2011)} was the construction of a ``parametrized Whitehead torsion''
\[\tau\colon \calS_\infty(p)\to\sections{\Omega\Wh_B(E)}{B}\]
whenever $p$ is a bundle of compact manifolds, such that the following holds:

\begin{thm}\label{main_result_of_paper1}
The diagram
\[\xymatrix{
{\calS_\infty(p)} \ar[rr]^\tau \ar[d]^\alpha && {\sections{\Omega\Wh_B(E)}{B}} \ar[d]^\beta\\
{\calS_\infty(E)} \ar[rr]^\tau && {\Omega\Wh(E)}
}\]
is a weak homotopy pull-back, with $\beta$ as in section \ref{section:definition}.
\end{thm}

Moreover, if $M$ is a compact topological manifold, the map
\[\pi_0(\tau)\colon \calS_\infty(M)\to\pi_1\Wh(M)\cong\Wh(\pi_1 M)\]
agrees with the classical Whitehead torsion, sending the class of a homotopy equivalence $f\colon N\to M$ to its Whitehead torsion $\tau(f)$.

\begin{proof}[Proof of Theorem \ref{main_result_1}]
Assumption (i) is clearly necessary while assumption (ii) is necessary by Theorem \ref{fqm:RRwc}.

Now suppose that assumptions (i) and (ii) hold, such that we can factor $f=p'\circ \lambda'$ where $p'$ is a fiber bundle of compact topological manifolds and $\lambda'$ a homotopy equivalence. Denote by $F'_b$ the fiber of $p'$ over $b$ and consider the following commutative diagram
\begin{equation}\label{fqm:comparison_diagram}
\xymatrix{
{\pi}_0\calS_\infty(p') \ar[d]^{\pi_0(\alpha')} \ar[r]^(.4)\tau & H^0(B;\Wh(F'_b)) \ar[d]^{\pi_0(\beta')}  \\
{\pi}_0\calS_\infty(E') \ar[r]^(.4)\tau & {\Whp{E'}} 
}
\end{equation}
which is $\pi_0$ of the pull-back square from Theorem \ref{main_result_of_paper1}, applied to the bundle $p'$. 

By Lemma \ref{fqm:stable_question_formulated_with_structure_spaces}, $f$ is homotopic to a bundle of compact manifolds if and only if the element defined by $\lambda'$ in the lower left-hand corner comes from an element in the upper left-hand corner. Using the pull-back property, this is equivalent to saying that the corresponding element $\tau([\lambda'])$ in the lower right-hand corner comes from an element in the upper right-hand corner. Thus, if we define $o(f)$ as the class of $\tau([\lambda'])$ in the cokernel of $\beta'$, $f$ fibers stably if and only if $o(f)=0$. As the fibrations $p$ and $p'$ are fiber homotopy equivalent, the cokernels of $\pi_0(\beta)$ and $\pi_0(\beta')$ are isomorphic. So we may think of $o(f)$ as an element in the cokernel of $\pi_0(\beta)$.

Finally we have to show that $o(f)$ is well-defined. Indeed suppose that we choose another factorization $f=\bar p'\circ\bar\lambda'$ with $\bar p'\colon \bar E'\to B$ a fiber bundle of compact manifolds. Then by the composition rule the resulting torsion changes by the torsion of $\lambda'\circ \bar\lambda'^{-1}\colon \bar E'\to  E'$, which is in the image of $\alpha'$ since it comes from a fiber homotopy equivalence. Thus, when passing to the cokernel of $\pi_0(\beta)$, the element $o(f)$ is not affected.
\end{proof}

\begin{proof}[Proof of Theorem \ref{main_result_2}]
We saw in Lemma \ref{fqm:stable_question_formulated_with_structure_spaces} that the set $C/{\sim}$ is in bijection with $\pi_0(\alpha')^{-1}([\lambda])$, which by square \eqref{fqm:comparison_diagram} is in bijection to $\pi_0(\beta')^{-1}[\tau(\lambda)]$ and thus to the kernel of $\pi_0(\beta')$ as $\beta'$ is an infinite loop map. Now use that the kernels of $\pi_0(\beta)$ to $\pi_0(\beta')$ are isomorphic.
\end{proof}

\section{Change of base and total space}

The two problems of ``change of base'' and ``change of total space'' are interesting special cases where the parametrized Wall obstruction does not play a role. We first consider them in the light of the general theory. After that we offer a second, more geometric perspective using families of $h$-cobordisms. This second perspective makes it easier to find an estimate for a stable range.

\begin{thm}[Change of total space]\label{fqm:change_of_total_space}
Let $p\colon M\to B$ be a fiber bundle of compact topological manifolds over a compact topological manifold, and let $N$ be another compact topological manifold, equipped with a homotopy equivalence $f\colon N\to M$:
\[\xymatrix{
N\ar[rr]^f_\simeq \ar[rrd]_{pf} && M \ar[d]^p\\
&& B
}\]
Then $pf$ stably fibers if and only if the Whitehead torsion $\tau(f)$ lies in the image of 
\[\pi_0(\beta)\colon H^0(B;\Omega\Wh(F_b))\to \Whp{M}\]
for $p$.
\end{thm}

\begin{thm}[Change of base]\label{fqm:change_of_base}
Let $p\colon M\to B$ be a fiber bundle of compact topological manifolds over a compact topological manifold, and let $C$ be another compact topological manifold, equipped with a homotopy equivalence $f\colon B\to C$:
\[\xymatrix{
M \ar[d]_p \ar[rrd]^{fp} \\
B \ar[rr]^f_\simeq && C
}\]
Then $fp$ stably fibers if and only if the image of the Whitehead torsion $f_*^{-1}\tau(f)\in\Whp{B}$ under the transfer homomorphism 
\[p^*\colon \Whp{B}\to\Whp{M} \]
lies in the image of $\pi_0(\beta)$. 

In particular, if the fiber $F$ of $p$ is connected, $\pi_1(B)$ acts trivially on $F$, and $fp$ stably fibers, then 
\[\chi_e(F)\cdot \tau(f)=0\in\Whp{C}.\]
\end{thm}

\begin{proof}[Proof of Theorem \ref{fqm:change_of_total_space}]
Notice that $pf$ is already a factorization into a homotopy equivalence followed by a fiber bundle. So, conditions (i) and (ii) of Theorem \ref{main_result_1} are satisfied, and the torsion obstruction $o(f\circ g)$ is just the image of the Whitehead torsion of $f$ in the cokernel.
\end{proof}

\begin{proof}[Proof of Theorem \ref{fqm:change_of_base}]
Denote by $k\colon C\to B$ a homotopy inverse of $f$, and consider the pull-back 
\[\xymatrix{
k^*M \ar[rr] \ar[d]^{k^*p} && M \ar[d]^{p}\\
C \ar[rr]^k && B
}\]
Now $f$ induces a map $\bar f\colon M\to k^*M$ such that $f\circ p=k^* p\circ\bar f$ is a factorization of $f\circ p$ into a homotopy equivalence followed by a fiber bundle. Thus $o(f\circ p)$ is given by the class of $\bar f^{-1}_* \tau(\bar f)$, which satisfies
\[\bar f^{-1}_* \tau(\bar f)= p^* f^{-1}_*\tau(f)\]
by the geometric definition of the transfer map \cite{Anderson(1974)}. 

Now suppose that $F$ is connected and $\pi_1(B)$ acts trivially. In this case the composite $p_*\circ p^*$ is just multiplication with the Euler characteristic of $F$ \cite{Lueck(1987)}. We saw that if $f\circ p$ stably fibers, then $p^* f_*^{-1}\tau(f)$ comes from some element $\kappa\in\Whp{F}$ under the map induced by the inclusion $i\colon F\to M$. As the composite $p\circ i$ is nullhomotopic, we have
\[0=p_* i_* \kappa = p_* p^* f_*^{-1} \tau(f)=\chi_e(F) \cdot f_*^{-1}\tau(f)\in\Whp{B}.\qedhere\]
\end{proof}

Using the relation between the parametrized torsion and higher $h$-cobor\-dism theory, we now give a second perspective on the change of total space problem. Under smoothability conditions, this approach allows to estimate a stable range using the stability results of Igusa \cite{Igusa(1988)}.

In the change of total space problem as in Theorem \ref{fqm:change_of_total_space}, suppose for simplicity that $B$ is connected. Denote by $k$ the smallest dimension of a CW complex homotopy equivalent to $B$, and by $n$ the the smallest dimension of a CW complex homotopy equivalent to the fibers.

\begin{thm}\label{fqm:change_of_total_space_revisited}
In the situation of Theorem \ref{fqm:change_of_total_space}, suppose that $M$, $N$, and the fibers of $f$ are smoothable. If $\tau(f)$ is in the image of $\pi_0(\beta)$, then the composite
\[\bar N:=N\times I^l\xrightarrow{\Proj} N\xrightarrow{f} M\xrightarrow{p} B\]
fibers as soon as
\begin{multline*}
\dim \bar N\geq\max\{2(n+k)+1, \dim M+n+k, \dim M+k+2,\\
 \dim B+2k+6, \dim B+3k+2, \dim N+3\}.
\end{multline*}
\end{thm}

\begin{proof}
The first step is to replace $f\colon N\to M$ by $\bar f\colon N\to\bar M$ which is a stably tangential homotopy equivalence, i.e.~$\bar f^*T\bar M\cong TN$ stably. 

Therefore recall that $M$ and $N$ are supposed to be smoothable, so we may choose a vector bundle reduction
\[(f^{-1})^*TN - TM\colon M\to BO\]
of the topological tangent bundle. It actually has a further reduction to a $O(n+k)$-bundle, the inclusion $BO(N)\to BO$ being $N$-connected. Let therefore $q\colon \bar M\to M$ be a disk bundle of this $(n+k)$-dimensional vector bundle. We obtain
\[T\bar M\vert_M\cong TM\oplus T_{\fib}q\vert_M\cong TM\oplus q\cong  (f^{-1})^*TN\]
stably, so if we let $\bar f\colon N\to \bar M$ be $f$ followed by the zero-section, then $\bar f$ is a stable tangential homotopy equivalence.

Notice that the map $\bar p=p\circ q$ still is a fiber bundle of compact smoothable manifolds (with a fiber we denote by $\bar F_b$), and that
\[\dim \bar M = \dim M+(n+k).\]

Now suppose that $\tau(\bar f)=\tau(f)$ disassembles, i.e.~that there is an element $\tau\in H^0(B;\Omega\Wh(\bar F_b))$ such that that $\pi_0(\alpha)(\tau)=\tau(f)$. By \cite[Corollary 9.3]{Steimle(2011)}, the element $\tau$ is the parametrized torsion obtained from glueing a fiberwise $h$-cobordism $E$ along the vertical boundary bundle $\partial \bar p$, provided that both 
\[K+1\geq k,\quad l-1\geq k\]
where $K$ is the concordance stable range of $\partial \bar F_b$ and $l$ is the connectivity of the pair $(\bar F_b, \partial \bar F_b)$. The following lemma (whose proof is an exercise using the Blakers-Massey theorem) shows that $l\geq n+k-1$. 

\begin{lem}\label{fqm:increasing_connectivity_by_stabilization}
If $\bar F\to F$ is an $L$-disk bundle over a compact manifold, and the pair $(F,\partial F)$ is $N$-connected, then $(\bar F, \partial \bar F)$ is $(N+L)$-connected.
\end{lem}

By our assumptions the manifold $\partial \bar F_b$ is smoothable, so Igusa's stability result \cite{Igusa(1988)} (see \cite[Theorem 1.3.4]{Weiss-Williams(2001)} for the topological range) says that $k-1\leq K$ whenever
\[\dim \partial \bar F_b\geq\max\{(2(k-1)+7, 3(k-1)+4\}.\]
Thus, stabilizing further if necessary, we obtain a parametrized $h$-cobordism $E$ over $\partial \bar p$ such that the torsion of the projection
\[\bar{\bar M}:=\bar M \cup_{\partial\bar p} E\to \bar M\]
is precisely $\tau$. Notice that
\[\dim \bar{\bar M} = \max\{\dim M+n+k, \dim M+k+2, \dim B+2k+6, \dim B+3k+2\}.\]

By the composition rule, the composite
\[\bar{\bar f}\colon N\to \bar M\to \bar {\bar M}\]
has torsion zero (and is still stably tangential). Now stabilize $N$ to obtain $\bar N:=N\times I^l$ such that $l\geq 3$ (so that $\bar N$ is $\pi$-$\pi$) and $\dim \bar N\geq 2(n+k)+1$, and stabilize either $\bar N$ or $\bar{\bar M}$ further so that the dimensions agree. Letting $K$ be a finite $(n+k)$-dimensional CW complex simple homotopy equivalent to $\bar N$, it follows that both $\bar N$ and $\bar{\bar M}$ define thickenings of $K$ in the sense of Wall \cite{Wall_IV}. It is known that stably, thickenings are classified by their tangent bundle. Now the dimension of the thickenings we consider exceeds $2(n+k)$, so we are in the stable range. But $\bar{\bar g}$ is stably tangential, hence the thickenings agree. Thus $\bar{\bar g}$ is homotopic to a homeomorphism.

Summarizing all the necessary stabilizations, we see that $pf$ fibers as soon as 
\begin{multline*}
\dim \bar N\geq\max\{2(n+k)+1, \dim N+3, \dim M+n+k,\\ 
\dim M+k+2, \dim B+2k+6, \dim B+3k+2\}.\qedhere
\end{multline*}
\end{proof}

\section{Examples I: Elementary applications}

In this section we give some immediate applications of our results on the stable fibering problem. The first one characterizes simple homotopy equivalences between compact manifolds as the homotopy equivalences that stably fiber. After that we give some implications for closed manifolds.

\begin{prop}\label{fqm:characterization_of_she}
A homotopy equivalence $f\colon M\to N$ between compact manifolds stably fibers if and only if $\tau(f)=0$. 

If $M$ and $N$ are closed smoothable of dimension $k$ and $\tau(f)=0$, then $f$ fibers after at most $\max\{2k+6, 3k+2\}$ stabilizations.
\end{prop}

\begin{proof}[Proof of Proposition \ref{fqm:characterization_of_she}]
This is a simple application of our results in the following change of total space problem:
\[\xymatrix{
M \ar[rr]^f_\simeq \ar[rrd]_f && N \ar[d]^{\id}\\
&& N
}\]
As the fibers of the identity are contractible, their Whitehead group vanishes. So $\pi_0(\beta)$ is the zero map and its cokernel is just $\Whp{M}$. Hence $o(f)=f_*^{-1}\tau(f)\in\Whp{M}$.
\end{proof}

Now we turn to closed manifolds and consider the change of total space problem
\[\xymatrix{
M \ar[rr]^f_\simeq \ar[rrd]_g && N\ar[d]^p\\
&& B
}\]
If $g$ stably fibers, i.e.~for some $n\gg 0$, $M\times D^{n+1}\to B$ fibers, then we may restrict to the boundary to see that
\[\bar g\colon M\times S^n\to M\xrightarrow{g} B\]
fibers for large enough $n$. So our theory gives sufficient conditions for $M\times S^n$ to fiber over $B$.

On the other hand, if $\bar g$ fibers, then it certainly stably fibers. It follows:

\begin{prop}
\begin{enumerate}
\item A necessary condition for $M\times S^{2N}\to B$ to fiber for large $N$ is that
\[2 o(g)=0.\] 
\item A sufficient condition for $M\times S^{2N}\to B$ to fiber for large $N$ is that
\[o(g)=0.\]
\item The sufficient condition is not necessary in general.
\end{enumerate}
\end{prop}

\begin{proof}
(i) We have $\tau(f\times\id_{S^{2N}})=\chi(S^{2N})\cdot \tau(f)= 2\tau(f)$, and its class in the cokernel of $\pi_0(\alpha)$ defines the obstruction $o(\bar g)$ for $\bar g$ to stably fiber.

(ii) Apply the results on the change of total space problem and restrict to the boundary.

(iii) Choose a homotopy equivalence $f\colon M\to K\times S^1$ between closed manifolds such that $\tau(f)\neq 0$ and $2\tau(f)=0$, and let $p\colon K\times S^1\to S^1$ denote the projection. Hence $o(pf)\neq 0$; in contrast, if $q\colon M\times S^{2N}\to M$ is the projection, then $o(pfq)=0$. 

We will show in Theorem \ref{fqm:comparison_to_FLS_obstructions} that for $B=S^1$, the stable fibering obstruction $o(pfq)$ and the obstructions $\tau_{\fib}(pfq)$ defined in \cite{Farrell-Lueck-Steimle(2009)} agree. Now Farrell's fibering theorem together with the comparison of the different obstructions \cite[Theorem 8.1]{Farrell-Lueck-Steimle(2009)} shows that $pfq$ fibers.
\end{proof}

\begin{prop}
If $f\colon M\to B$ is any map between closed manifolds whose homotopy fiber is finitely dominated, then the composite
\[M\times S^1\times S^1\times S^N\xrightarrow{Proj} M\xrightarrow{f} B\]
fibers for large enough $N$.
\end{prop}

\begin{proof}
The parametrized Wall obstruction becomes zero after taking product with $S^1$ (see e.g.~\cite[Corollary 5.2.5]{Weiss-Williams(2001)}). Hence for the map $M\times S^1\to B$, the fibering obstruction is defined. As it is given by a Whitehead torsion, it becomes zero after taking product with another $S^1$. Therefore $M\times S^1\times S^1\times D^{N+1}\to B$ fibers. Now restrict to the boundary.
\end{proof}

\section[Examples II: Stable vs.~unstable and block fibering and TOP vs.~DIFF]{Examples II: Stable vs.~unstable and block fibering and TOP vs.~DIFF}

In this section we give examples of maps $f\colon M\to B$ that fiber stably but not unstably. Of course, if the dimension of $M$ and $B$ agree then fibering $f$ unstably just means deforming the map to a homeomorphism, whereas $f$ stably fibers if and only if it is a simple homotopy equivalence (Proposition \ref{fqm:characterization_of_she}): This gives obvious examples. 

More interestingly, we will consider two types of situations of arbitrarily high codimension. The first one considers the tangential data; supposing that a recent conjecture of Reis-Weiss on topological rational Pontryagin classes holds, we obtain a lower bound on the number of stabilizations needed. The second one applies surgery theory to produces an example which actually does not even block fiber.

We also consider certain maps to spheres with spherical fibers where unstable fibering and block fibering are equivalent and we expand on an example of Klein-Williams to produce examples that fiber stably in TOP but not in DIFF.

\subsection{Bundle theory}

Let $Z$ be an exotic complex projective space equipped with a homotopy equivalence
\[h\colon Z\to \CC P^{2n+1}\]
such that, for some $k\neq 0$, the $\calL$-genus satisfies
\begin{equation}\label{fqm:condition_on_L_genus}
(h^*)^{-1}\calL(Z) = \calL(\CC P^{2n+1})\cdot (1+8ke^{2n})\in H^*(\CC P^{2n+1})\cong \QQ[e]/(e^{2n+2}). 
\end{equation}
(Take all the cohomology rings with rational coefficients.) We will show below that such objects exist.

\begin{conj}[{\cite{Reis-Weiss(2011)}}]\label{fqm:conjecture_on_pontryagin_classes}
If $\xi$ is a $\TOP(n)$-bundle over $B$, then the $i$-th rational Pontryagin class $p_i(\xi)\in H^{4i}(B;\mathbb{Q})$ vanishes provided $i>n/2$.
\end{conj}

\begin{prop}
\begin{enumerate}
\item The composite
\[Z\times S^N \xrightarrow{\pi} Z\xrightarrow{h} \CC P^{2n+1}\]
of $h$ with the projection fibers stably. If $Z$ is smoothable then it fibers even unstably whenever $N\geq 12n+7$.
\item If Conjecture \ref{fqm:conjecture_on_pontryagin_classes} holds, then for $N\leq 2n-1$, the map from (i) does not fiber (unstably).
\end{enumerate}
\end{prop}

\begin{proof}
(i) This is an application of the change of total space problem.

(ii) Let $p\colon Z\times S^N\to\CC P^{2n+1}$ be a fiber bundle homotopic to the map of (i). Then
\[T(Z\times S^N)\cong p^* T(\CC P^{2n+1})\oplus \eta\]
for an $N$-dimensional bundle $\eta$. Hence
\[\calL(Z\times S^N)=p^*\calL(\CC P^{2n+1})\cdot \calL(\eta).\]

But $\calL(Z\times S^N)=\pi^*\calL(Z)=p^*(h^*)^{-1}\calL(Z)$ as the sphere is stably parallelizable. It follows that
\[\calL(\eta) = p^*\bigl(\calL(\CC P^{2n+1})^{-1} \cdot (h^*)^{-1}\calL(Z)\bigr)=1+8kp^*(e)^{2n}\]
using \eqref{fqm:condition_on_L_genus} for the last equality. Hence, since $p^*$ is injective, the $\calL$-genus of $\eta$ is non-zero in degree $4n$. Inductively one concludes
\[p_i(\eta)=0 \;(i<n),\quad p_n(\eta)\neq 0\]
using the fact that the coefficient of $p_i$ in $\mathcal L_i$ is non-zero for all $i$ \cite[I.1.(11)]{Hirzebruch(1978)}. So $\eta$ must be at least $2n$-dimensional: $N\geq 2n$.
\end{proof}

We now indicate why a homotopy equivalence
\[h\colon Z\to \CC P^{2n+1}\]
with the property \eqref{fqm:condition_on_L_genus} exists. 
This construction is due to Madsen-Milgram \cite{Madsen-Milgram(1979)}.
Let $f\colon X\to \CC P^{2n}$ be a topological degree one normal map corresponding to the composite
\[\CC P^{2n}\to \CC P^{2n}/\CC P^{2n-1} \cong S^{4n} \to G/\TOP\]
where the last map represents $k$ times a generator of $\pi_{4n}(G/\TOP)\cong \ZZ$. Let $E\to \CC P^{2n}$ be the disk bundle of the tautological vector bundle. We may pull back the normal map $f$ to $E$. By the $\pi$-$\pi$-theorem, this pulled-back normal map is cobordant to a map $g\colon Y\to E$ which is a homotopy equivalence and restricts to a homotopy equivalence over the boundary $\partial E=S^{4n+1}$. The Poincar\'e conjecture implies that $\partial Y\cong S^{4n+1}$ homeomorphically. Thus we may cone off $g$ at the boundaries to obtain a topological manifold $Z$ and a homotopy equivalence $h\colon Z\to \CC P^{2n+1}$.

We have to show that \eqref{fqm:condition_on_L_genus} holds. To do that, we will use the characteristic classes
\[K_{4n}\in H^{4n}(G/\TOP;\QQ)\]
given uniquely by the property that if $\gamma\colon M\to G/\TOP$ is a normal invariant on a closed $4k$-manifold, then its simply-connected surgery obstruction is given by the formula \cite[Theorem 4.9]{Madsen-Milgram(1979)}
\[s(M,\gamma)=\big\langle\calL(M)\cdot\big(\sum_{i\geq 1} \gamma^*(K_{4i})\big), [M]\big\rangle.\] 

Now the surgery obstructions of $h\vert_{\CC P^i}$ for $i<2n$ are zero while the surgery obstruction of $h\vert_{\CC P^{2n}}$ is $k$, hence inductively one concludes that 
\[\gamma^*(K_{4i})=0\quad(4i< 2n),\quad \gamma^*(K_{4n})=k e^{2n}\in H^{4n}(\CC P^{2n+1}).\]

Denote by $\calL\in H^*(B\TOP;\QQ)$ the $\calL$-class of the universal bundle. By \cite[Corollary 4.22]{Madsen-Milgram(1979)}, its restriction along $G/\TOP\to B\TOP$ is given by $1+8K$, where $K=K_4+K_8+\dots$.
Hence
\[(h^*)^{-1}\calL(Z) = \calL(\CC P^{2n+1}) \cdot \gamma^*(1+8K) = \calL(\CC P^{2n+1}) \cdot (1+8ke^{2n}),\]
as claimed. 

\begin{rem}
To the knowledge of the author, it is unknown in general which of these fake $\CC P^{2n+1}$ are smoothable. The following argument shows that there are infinitely many smoothable examples for even $n$. 

For each $n$ there exists a number $A_n$ such that the normal invariant $f$ is smoothable if and only if $k$ is a multiple of $A_n$. (In fact $A_n$ is the order of the generator of $\pi_{4n}(G/\TOP)$ in the torsion group $\pi_{4n-1}(\TOP/O)$.) Hence the subgroup of all smooth normal invariants of $\CC P^{2n}$ satisfying \eqref{fqm:condition_on_L_genus} is infinite.

Using the $\pi$-$\pi$-theorem in the smooth setting, we obtain a map
\[ [\CC P^{2n}, G/O] \cong [E, G/O] \cong \calS^\DIFF(E) \to \calS^\DIFF(\partial E) = \Gamma_{4n+1}\]
from the smooth normal invariants of $\CC P^{2n}$ to the group of homotopy spheres. By Brumfiel \cite[Corollary 6.6]{Brumfiel(1971)}, this map is a group homomorphism if $n$ is even. Hence in this case it has an infinite kernel. 

If $f\colon X\to\CC P^{4n}$ is a smooth normal map of degree one which represents an element in the kernel, this means the following: The pull-back of $f$ to $E$ is cobordant to a normal map $g\colon Y\to E$ which restricts to a diffeomorphism on the boundary. In this case the coning procedure yields a homotopy equivalence $h\colon Z\to \CC P^{4n+1}$ where $Z$ is smooth.
\end{rem}

\subsection{Surgery theory}

Now we come to the surgery-theoretic example. Let $\gamma\colon X\to G/\TOP$ be a normal invariant on a closed manifold $X$, let $M$ be a closed manifold and let $h\colon M\to X\times S^k$ be a simple homotopy equivalence which, considered as a normal invariant, restricts to $\gamma$ over $X\times \{*\}$. (Such a simple homotopy equivalence can be obtained as follows: Pull-back of $\gamma$ defines a normal invariant on $X\times D^{k+1}$ which, by the $\pi$-$\pi$ theorem, can be represented by a simple homotopy equivalence of pairs $(k,h)\colon (N,M)\to X\times (D^{k+1}, S^k)$.)

\begin{prop}\label{fqm:surgery_example}
If the surgery obstruction of $\gamma$ is non-zero, then the composite 
\[f\colon M\xrightarrow{h} X\times S^k\xrightarrow{\Proj} S^k\]
does not fiber. It always fibers stably. 
\end{prop}

\begin{proof}
The composite fibers stably by Proposition \ref{fqm:change_of_total_space}. Suppose there exists a fiber bundle $F\to M\xrightarrow{p} S^k$ homotopic to $f$. Then we can lift the homotopy to obtain a homotopy
\[H\colon M\times I\to X\times S^k\times I\]
over $S^k$ which restricts to $h$ at $M\times 0$ and which is a fiber homotopy equivalence $F\to X\times S^k$ over 1. Taking a transverse preimage over the base point of $S^k$ yields a degree one normal cobordism
\[(W; N, F) \to (X\times S^k\times I; X\times S^k\times 0, X\times S^k\times 1)\]
whose restriction over $1$ is a homotopy equivalence. The restriction over $0$ corresponds to the element $\gamma$, which therefore has surgery obstruction 0, contradicting the assumption. 
\end{proof}

\subsection{Spherical fibrations over spheres}

\begin{prop}\label{fqm:example_with_spheres_1}
Let $f\colon M\to S^{2k}$ be a map between closed topological manifolds whose homotopy fiber is a $2n$-sphere, where $2n\geq 4k$. Suppose that $\dim M\geq 6$. Then,
\begin{enumerate}
\item $f$ always fibers stably.
\item The following are equivalent:
	\begin{enumerate}
	\item $f$ fibers unstably,
	\item $f$ block-fibers,
	\item for a regular preimage $F$ of a point in $S^{2k}$, the degree one normal map $F\to S^{2n}$ given by the inclusion of the fiber into the homotopy fiber has surgery obstruction zero.
	\end{enumerate} 
\end{enumerate}
\end{prop}

\begin{rem}
Casson \cite{Casson(1967)} showed that (b) and (c) are equivalent in the smooth case; in the topological case this equivalence can be deduced from Quinn's thesis \cite{Quinn(1970)}.
\end{rem}

\begin{add}
If, in the situation of Proposition \ref{fqm:example_with_spheres_1}, the number $n=2l$ is even, the following conditions are equivalent and equivalent to the ones from part (ii):
\begin{enumerate}\setcounter{enumi}{3}\renewcommand{\labelenumi}{(\alph{enumi})}
\item For a regular preimage $F$ of a point in $S^{2k}$, the signature $\sigma(F)$ is zero,
\item the $l$-th Pontryagin class $p_l(M)\in H^{4l}(M;\mathbb{Q})$ is zero.
\end{enumerate}
\end{add}

The proof of Proposition \ref{fqm:example_with_spheres_1} will make use of the following lemma, which is due to Siebenmann \cite[Essay V, \S 5]{Kirby-Siebenmann(1977)} for $n\neq 4$  and to Freedman-Quinn \cite[Theorem 8.7A]{Freedman-Quinn(1990)} for $n=4$:

\begin{lem}
The stabilization map $B\TOP(n)\to B\TOP$ is $n$-connected.
\end{lem}

\begin{proof}[Proof of Proposition \ref{fqm:example_with_spheres_1}]
(i) Let $f=p\circ\lambda$ be a factorization into a homotopy equivalence followed by a fibration. Since $\pi_{2k-1} G/\TOP=0$ and we are in the stable range, we may assume that $p$ is the sphere bundle of a $\TOP(2n+1)$-bundle $\eta$. Hence we are in a change of total space situation, and the result follows since $E$ is simply-connected.

(ii) The implication from (a) to (b) is obvious. The argument for the implication from (b) to (c) is very similar to the proof of Proposition \ref{fqm:surgery_example}: a homotopy from $f$ to the projection map of a block bundle $g$ induces a normal bordism between the degree one normal map $F=f^{-1}(*)\to S^{2n}$ and the identity map on $S^{2n}$. 

Suppose now that (c) holds. Again factor $f=p\circ \lambda$ as in (i). Again we can and will assume that $p$ is the sphere bundle of a $\TOP(2n+1)$-bundle $\eta$. 

By the dimension assumptions, $p$ has a section $s\colon S^{2k}\to E$, up to homotopy. Denote by $i\colon S^{2n}\to E$ the inclusion of the homotopy fiber. Collapsing the lower skeleta defines a map
\[\pi\colon E\to S^{2n+2k};\]
homotopy-theoretic calculations show that the sequence
\[[S^{2n+2k}, G/\TOP] \xrightarrow{\pi^*} [E, G/\TOP] \xrightarrow{i^*\oplus s^*} [S^{2n}, G/\TOP]\oplus [S^{2k}, G/\TOP]\]
is exact.

\begin{lem}\label{fqm:lemma_split_by_surgery_obstruction}
The map $\pi^*$ is split by the surgery obstruction 
\[ [E, G/\TOP]\to L_{2n+2k}(\ZZ)\cong [S^{2n+2k}, G/\TOP].\]
\end{lem}

Let $x:=j(\lambda)$ where
\[j\colon \calS(E)\to [E, G/\TOP]\]
is the canonical map from the surgery structure set to the set of normal invariants. As the surgery obstruction map $[S^{2n}, G/\TOP]\to L_{2n}(\mathbb{Z})$ is an isomorphism, condition (c) says that $i^*(x)=0$. The element $s^*(x)$ is represented by some $\TOP(2k+1)$-bundle $\xi$ over $S^{2k}$ and a fiber homotopy trivialization $t\colon S(\xi)\to S(\varepsilon)$ of the corresponding sphere bundle.

Since $2n\geq 4k$, we can split off the bundle $\eta$ a $(2k+1)$-dimensional trivial bundle $\varepsilon$, so that $\eta\cong \eta'\oplus\varepsilon$ and we can consider the homotopy equivalence
\[\mu\colon S(\eta'\oplus\xi)\to S(\eta'\oplus\varepsilon)\cong S(\eta)=E\]
induced by the identity on $\eta'$ and the trivialization $t$. 

\begin{lem}\label{fqm:lemma_on_transfer_for_normal_invariants}
The homotopy equivalence $\mu$, considered as a normal invariant on $E$, agrees with $p^*s^*(x)$.
\end{lem}

Now, $p^*s^*(x)$ and $x$ agree after applying $s^*$ but also after applying $i^*$ (when both are zero). It follows from the exact sequence above that the difference $p^*s^*(x)-x$ lifts over the map $\pi^*$, which is split by the surgery obstruction. But the surgery obstruction of both $x$ and $p^*s^*(x)$ is zero since they are represented by homotopy equivalences. Hence $p^*s^*(x)=x$. By the surgery exact sequence, the map $j$ is injective, so the manifold structures $\lambda$ and $\mu$ agree as well. But $p\circ\mu$ is a fiber bundle projection.
\end{proof}

\begin{proof}[Proof of Lemma \ref{fqm:lemma_split_by_surgery_obstruction}]
Let $\alpha\colon S^{2n+2k}\to G/\TOP$ represent a generator, corresponding to a degree one normal map $g\colon N\to S^{2n+2k}$. We claim that the normal invariant $\pi^* g$ is represented by the degree one normal map ${\id}\sharp g\colon E\sharp N\to E\sharp S^{2n+2k}=E$. In fact if $\alpha$ is given by the $\TOP$-bundle $\xi$ and a proper fiber homotopy trivialization $t\colon \xi\to\varepsilon$, then $g$ is obtained by taking a regular preimage of $S^{2n+2k}\subset \varepsilon$ under $t$ (see \cite[Theorem 2.23]{Madsen-Milgram(1979)}). 

We may assume that $t$ is a bundle isomorphism over the lower hemisphere which contains the point $\infty$ and that around $\infty$, $g$ is the homeomorphism from the zero section in $\xi$ to the zero section in $\varepsilon$. Since $\pi$ maps $E-D^{2n+2k}$ to the point $\infty\in S^{2n+2k}$, the map $q^* t\colon q^*\xi\to q^*\varepsilon$ is then already transverse to $E-D^{2n+2k}$ and $E\sharp N$ is a regular preimage.

Since the simply-connected surgery obstruction is additive under connected sum and the surgery obstruction of $f$ is a generator in $L_{2n+2k}(\ZZ)$, it follows that the surgery obstruction of $\pi^*\alpha$ is also a generator.
\end{proof}

\begin{proof}[Proof of Lemma \ref{fqm:lemma_on_transfer_for_normal_invariants}]
The homotopy equivalence $\mu$ extends to a homotopy equivalence
\[\bar\mu\colon D(\eta'\oplus\xi)\to D(\eta'\oplus\varepsilon)\]
between disk bundles, which is the pull-back of the homotopy equivalence $\bar t\colon D(\xi)\to D(\varepsilon)$ under the bundle projection $D(\eta'\oplus\varepsilon)\to D(\varepsilon)$. 

The restriction of $\bar t$ over a regular preimage of $S^{2k}\in D(\varepsilon)$ is classified by $s^*(x)\colon E\to G/\TOP$. Since the inclusion $S^{2k}\subset D(\varepsilon)$ is a homotopy equivalence, the normal invariant $\bar t$ is classified by the composite of $s^*(x)$ with the projection $D(\varepsilon)\to S^{2k}$. From the pull-back property it follows that $\bar\mu$ is classified by the composite
\[D(\eta)=D(\eta'\oplus\varepsilon)\to S^{2k}\xrightarrow{s^*(x)} G/\TOP.\]
Now restrict to the boundary.
\end{proof}

\begin{proof}[Proof of Addendum]
Suppose that $f$ is a fiber bundle with fiber $S^{4l}$. By the Poincar\'e conjecture, the fiber is actually a sphere. We can extend $f$ to a disk bundle $\bar f\colon \bar M\to S^{2k}$. Now 
\[TM\cong T\bar M\vert_M\]
stably, so $p_l(M)$ is the restriction of $p_l(\bar M)\in H^{4l}(\bar M)\cong H^{4l}(S^{2k})=0$. Hence $p_l(M)=0$. 

Now suppose that $p_l(M)=0$. Denote by $F$ a regular preimage of a point $*\in S^{2k}$ under $f$. Then $F$ is framed in $M$, hence, for all $i$,
\[p_i(F)=p_i(M)\vert_F\]
and in particular $p_l(F)=0$. Moreover the inclusion $F\to M$ factors, up to homotopy, over $S^{4l}$, so all the Pontryagin classes of $F$ vanish except for $p_0$. Hence $\mathcal{L}(F)=1$. The Hirzebruch signature formula yields $\sigma(F)=0$.

Finally $\sigma(F)-\sigma(S^{4l})$ is a multiple of the surgery obstruction (in $L_{4l}(\ZZ)\cong \ZZ$) of the degree one normal map $F\to S^{4l}$. As $\sigma (S^{4l})=0$, the vanishing of $\sigma (F)$ is equivalent to the vanishing of the surgery obstruction.
\end{proof}

\subsection{DIFF vs.~TOP}

\begin{prop}\label{fqm:example_with_spheres_2}
Let $x\in\pi_{4i+1}^s$ such that $\eta^2 x$ is in the image of the $J$-homomor\-phism and let $p\colon E\to S^{4i+2}$ be a based $n$-spherical fibration ($n>4i+3$) corresponding to $x\in\pi_{4i+1+n}\cong \pi_{4i+2} BF_n$, where $F_n$ is the monoid of pointed self-homotopy equivalences of $S^n$. If $M$ is a compact smooth manifold and $\lambda\colon M\to E$ is a homotopy equivalence, then:
\begin{enumerate}
\item $p\circ \lambda$ fibers stably in TOP.
\item $p\circ \lambda$ does not fiber stably in DIFF.
\end{enumerate}
\end{prop}

\begin{proof}[Proof of Proposition \ref{fqm:example_with_spheres_2}]
(i) Since $n>4i+3$, the stabilization map 
\[\pi_{4i+1} G_n/\TOP_n\to\pi_{4i+1} G/\TOP\]
is an isomorphism. But $\pi_{4i+1} G/\TOP=0$. Hence the fibration $p$ is the sphere bundle of a $\TOP_n$-bundle and we are in a change of total space situation. Thus, $p\circ\lambda$ fibers stably as $E$ is simply-connected.

(ii) The following argument follows the lines of Klein-Williams \cite{Klein-Williams(2009)}. Suppose that $p\circ\lambda$ fibers stably in DIFF. Then by \cite{Dwyer-Weiss-Williams(2003)}, the parametrized $A$-theory characteristic 
\[\chi(p)\in H^0(S^{4i+2};A(S^n))\]
becomes zero in $H^0(S^{4i+2};\Wh^{\DIFF}(S^n))$. Calling $y$ the image of $\chi(p)$ under the projection
\[H^0(S^{4i+2};\Wh^{\DIFF}(S^n))\to H^0(S^{4i+2};\Wh^{\DIFF}(*))\cong \pi_{4i+2}\Wh^{\DIFF}(*)\]
we conclude that $y=0$.

As explained in \cite[\S 8]{Klein-Williams(2009)}, $y$ is the image of $x$ under the composite
\[\pi_{4i+2} BG\xrightarrow{F} \pi_{4i+2} A(*) \to \pi_{4i+2} \Wh^{\DIFF}(*)\]
where $F$ is the map defined by Waldhausen \cite{Waldhausen(1982)}. But it was shown in \cite{Boekstedt-Waldhausen(1983)} that the image of $x$ is non-zero.
\end{proof}

It would be interesting to have an example of a map that fibers in TOP where the smooth Wall obstruction is zero, but the smooth fibering obstruction is not. This would probably require a deeper analysis of the higher homotopy type of $\WhPL(F)$ and $\Wh^{\DIFF}(F)$ for a suitable $F$ whose fundamental group has non-vanishing Whitehead group.

\section{Examples III: Results of Chapman-Ferry}\label{fqm:section_examples_II}

The obstruction theory developed in this chapter allows us to re-interpret obstructions obtained by Chapman-Ferry \cite{Chapman-Ferry(1978)} on fibering compact $Q$-manifolds over compact ANRs. Explicitly, Chapman-Ferry deal with the cases where $B$ is a wedge of copies of $S^1$ and where $B$ is $n$-dimensional and the fibers are $n$-connected. We will see that in both cases the spectral sequence can be used to analyze our obstructions further.

By definition, a $Q$-manifold is a separable metric space which is locally homeomorphic to open subsets of the Hilbert cube $Q=\Pi_{i=1}^\infty I$. The fibering problem for $Q$-manifolds asks whether a given a map $f\colon M\to B$ from a compact $Q$-manifold to a compact ANR is homotopic to a fiber bundle projection whose fibers are compact $Q$-manifolds again. (Since the projection $Q\times I\to Q$ is homotopic to a homeomorphism, there is no difference between the stable and the unstable fibering problem.)

It will be shown in Appendix \ref{section:fibering_Q_manifolds} that the obstruction theory for fiber fibering compact $Q$-manifolds over a compact ANR agrees with the obstruction theory for compact topological manifolds developed so far. Hence all the fibering results may be applied in either context.

\begin{lem}\label{fqm:connectedness_of_Wh_X}
Let $p\colon E\to B$ be a fibration. If the base space is homotopy equivalent to a CW complex of dimension at most $n$, and all the fibers $F_b$ are $n$-connected, then $H^0(B;\Wh(F_b))=0$. Hence $\Wall(p)=0$ when defined.

In the case $n=1$, it is enough to suppose $\Whp{F_b}=0$. In the case $n=2$, it is enough to suppose that $F_b$ is 1-connected.
\end{lem}

\begin{proof}
If a map $X\to Y$ is $n$-connected, then the induced map $A(X)\to A(Y)$ is $n$-connected by \cite{Waldhausen(1978)} and so is $\Wh(X)\to\Wh(Y)$ by a 5-lemma type argument. As $\Wh(*)$ is contractible, it follows that $\Wh(X)$ is $n$-connected whenever $X$ is $n$-connected. If $X$ is 1-connected, then $\Wh(X)$ is 2-connected (see \cite[section 3]{Hatcher(1978)}, with the correction in \cite{Igusa(1984)}). Therefore the spectral sequence for $H^0(B;\Wh(F_b))$ has vanishing $E_2$ page. 
\end{proof}

If the fibration $p$ has a section, then this result can be considerably strenghtened:

\begin{thm}[{\cite{Klein-Williams(2009)}}]\label{fqm:thm_KW}
If $p$ has a section, all the fibers are $n$-connected, and $B$ is homotopy equivalent to a CW complex of dimension at most $2n$, then the map
\[H^0(B;A(F_b))\to H^0(B;\Wh(F_b))\]
is zero. Hence, $\Wall(p)=0$ when defined.
\end{thm}

\begin{exmpl}\label{fqm:two_connected_map_to_two_sphere}
A two-connected map $f\colon M\to S^2$ stably fibers. In fact, such a map is split up to homotopy, so Theorem \ref{fqm:thm_KW} applies. The homotopy fiber $F$ has the homotopy type of a CW complex by Milnor \cite{Milnor(1959)}. Moreover $H_*(F)$ is finitely generated. In fact, the $E^2$-term of the Atiyah-Hirzebruch spectral sequence
\[E^2_{pq}= H_p(S^2, H_q(F)) \Rightarrow H_{p+q}(M)\]
consists of two columns only, so there is an exact sequence
\[0 \to E^\infty_{2,n-1} \to E^2_{2,n-1} \to E^2_{0,n} \to E^\infty_{0,n} \to 0\]
with
\[E^2_{2,n-1}= H_{n-1}(F),\quad E^2_{0,n}=H_n(F).\]
The $E^\infty$-page is finitely generated (since $H_*(M)$ is); one concludes by induction that $H_n(F)$ is also finitely generated.

As $\pi_1(F)=0$, it follows \cite{Wall(1965a)} that $F$ is homotopy finite. So $\Wall(p)$ is defined and zero. Moreover $o(f)=0$ as $M$ is simply-connected.
\end{exmpl}

Let always $f\colon M\to B$ be a map either between compact topological manifolds, or from a compact $Q$-manifold to a compact ANR. In the $Q$-manifold setting, the following results are due to Chapman-Ferry.

\begin{prop}
Suppose $B$ is  homotopy equivalent to a finite $n$-complex ($n\geq 1$). If the homotopy fiber $F$ of $f$ is homotopy finitely dominated and $n$-connected, then the torsion obstruction $o(f)\in\Whp{M}$ is defined and vanishes if and only if $f$ stably fibers.

If $n=1$, we can replace the assumption ``$F$ $1$-connected'' by ``$\Whp{F}=0$''. If $n=2$, we can replace the assumption ``$F$ $2$-connected'' by ``$F$ $1$-connected''.
\end{prop}

\begin{proof}
By the assumptions and Lemma \ref{fqm:connectedness_of_Wh_X}, the parametrized Wall obstruction is defined and zero. Moreover, the map $\beta\colon H^0(B;\Wh(F_b))\to\Whp{E}$ is zero since it factors through $\Whp{F_b}=0$.
\end{proof}

\begin{prop}
\begin{enumerate}
\item Suppose that $B$ is homotopy equivalent to a wegde of $n$ copies of $S^1$ ($n\geq 1$), and suppose that the homotopy fiber $F$ of $f$ is connected and homotopy equivalent to a finite complex. Then, $\Wall(p)$ is an element of $\bigoplus_n \Whp{F}_{\alpha_n}$ (coinvariants under the action of the fiber transport along the corresponding copy of $S^1$).
\item The torsion obstruction (whenever defined) is an element  in the quotient 
\[\Whp{M}/ (n-1)\cdot\Whp{F}^{\pi_1 B}.\]
In particular, if $n=1$, the torsion obstruction lives in $\Whp{M}$.
\end{enumerate}
\end{prop}

\begin{proof}
(i) By the spectral sequence, there is an exact sequence
\[0 \to H^1(B;\pi_1\Wh(F_b)) \to H^0(B;\Wh(F_b)) \xrightarrow{e} H^0(B;\pi_0\Wh(F_b))\to 0,\]
and the image of $\Wall(p)$ under the edge homomorphism $e$ is given by the finiteness obstruction of the fiber, which is zero by assumption. Therefore, $\Wall(p)$ lifts to
\[H^1(B;\pi_1\Wh(F_b))\cong\bigoplus_n H^1(S^1;\pi_1\Wh(F_b))\cong\bigoplus_n\Whp{F_b}_{\alpha_n}.\]

(ii) The map $\beta$ factors as
\[H^0(B;\Omega\Wh(F_b))\to H^0(B;\Whp{F_b})\cong\Whp{F_b}^{\pi_1 B}\xrightarrow{\chi_e(B)\cdot i_*}\Whp{M},\]
where the first map is a surjection and the Euler characteristic of $B$ is $n-1$.
\end{proof}

\begin{prop}
Let $f,g\colon M\to B$ be two homotopic projections of bundles of compact manifolds. Suppose that $B$ is homotopy equivalent to a wedge of $n$ copies of $S^1$. Denote by $F$ the homotopy fiber of $f$, which is homotopy equivalent to the homotopy fiber of $g$. The obstruction group $A$ for $f$ and $g$ being equivalent (in the sense of the introduction) fits into the following exact sequence:
\[0\to \bigoplus_n \pi_2\Wh(F)_{\alpha_n} \to A \to \Whp{F}^{\pi_1 B}\xrightarrow{(n-1)\cdot i_*}\Whp{M}\]
\end{prop}

\begin{proof}
By Theorem \ref{main_result_2}, $A$ is given by the kernel of 
\[\beta\colon H^0(B;\Omega\Wh(F_b))\xrightarrow{\gamma} H^0(B;\Whp{F})\cong \Whp{F}^{\pi_1 B} \xrightarrow{(n-1)\cdot i_*}\Whp{M}.\] 
By the spectral sequence, we have
\[0\to H^1(B;\pi_2\Wh(F_b))\to H^0(B;\Omega\Wh(F_b))\xrightarrow{\gamma} H^0(B;\pi_1\Wh(F_b))\to 0.\]
Thus
\[\ker\gamma \cong H^1(B;\pi_2\Wh(F_b))\cong  \bigoplus_n \pi_2\Wh(F)_{\alpha_n}.\]
Since $\gamma$ is surjective, there is a short exact sequence
\[0\to\ker\gamma\to A\to\ker((n-1)\cdot i_*)\to 0.\]
The claim follows.
\end{proof}

\section[Comparison with the obstructions by Farrell-L\"uck-Steimle]{Comparison with the obstructions by  Farrell-L\"uck-Steimle}\label{fqm:section_comparison_to_FLS}

The content of the author's Diploma thesis \cite{Steimle(2007)} was to define Whitehead torsion obstructions to fibering a manifold over another manifold. See \cite{Farrell-Lueck-Steimle(2009)} for a published and extended version. The goal of this section is to compare these obstructions.

Given a map $f\colon M\to B$ of topological manifolds, factor as usual as $f=p\circ \lambda$, a homotopy equivalence followed by a fibration. In \cite{Farrell-Lueck-Steimle(2009)}, two obstructions for $f$ to be homotopic to a fiber bundle are defined:
\begin{enumerate}
\item An element $\theta(f)\in H^1(B;\Whp{M})$ which is defined whenever the homotopy fiber $F$ of $f$ is homotopy finite (an obvious necessary condition). It is defined by the rule that whenever $\gamma\colon S^1\to B$ is a loop in $B$, then under the restriction map
\[H^1(B;\Whp{M})\xrightarrow{\gamma^*}H^1(S^1;\Whp{M})\cong\Whp{M}\]
$\theta(f)$ maps to $i_*(\tau)$, where $\tau$ is the Whitehead torsion of the fiber transport on $p$ along $\gamma$ (choosing an arbitrary simple structure on the fiber $F$).
\item If $\theta(f)=0$, there is defined an element 
\[\tau_{\fib}(f)\in\coker(\Whp{F}\xrightarrow{\chi_e(B)\cdot i_*}\Whp{M})\]
where $i\colon F\to M$ is the inclusion of the homotopy fiber, and $\chi_e(B)$ denotes the Euler characteristic. It is defined as follows: Choose a simple structure on the homotopy fiber of $f$ and perform a certain construction (inductively over the cells of $B$) to obtain a simple structure on $E$. Then $\tau_{\fib}(f)$ is the image of the Whitehead torsion of $\lambda\colon M\to E$, which is well-defined in the quotient.
\end{enumerate}

\begin{thm}\label{fqm:comparison_to_FLS_obstructions}
\begin{enumerate}
\item The image of $\Wall(p)$ under the restriction
\[H^0(B;\Wh(F_b))\to H^0(\{b\}; \Wh(F_b))\cong \tilde K_0(\ZZ[\pi F_b])\]
is the finiteness obstruction of the fiber.
\item Suppose that $F$ is homotopy finite. The image of the Wall obstruction $\Wall(p)$ under the secondary homomorphism
\begin{multline*}
\ker\bigl(H^0(B;\Wh(F_b))\to H^0(B;\pi_0\Wh(F_b))\bigr)\\
\to H^1(B;\Whp{F_b})\xrightarrow{i_*} H^1(B;\Whp{M})
\end{multline*}
is $\theta(f)$.
\item Suppose that $\Wall(p)=0$. The definition of the map $\beta$ as a composite
\[H^0(B;\Omega\Wh(F_b))\to \Wh(\pi_1 F_b)\xrightarrow{\chi_e(B)\cdot i_*} \Wh(\pi_1 E)\cong \Wh(\pi_1 M)\]
induces a map
\[\coker(\pi_0(\beta))\to\coker \big(\Whp{F}\xrightarrow{\chi_e(B)\cdot i_*}\Whp{M}\big)\]
under which $o(f)$ maps to $\tau_{\fib}(f)$. In particular, if $\chi_e(B)=0$ or $\Wh(\pi_1 F)=0$, then 
\[o(f)=\tau_{\fib}(f)\in\Whp{M}.\]
\end{enumerate}

\begin{proof}
(i) and (ii) were proved in Theorem \ref{fqm:spectral_sequence}.

(iii) If $\Wall(p)=0$, then we may assume that $p$ is a bundle of compact topological manifolds, and it follows from \cite[Lemma 3.19]{Farrell-Lueck-Steimle(2009)} that the simple structure on $E$ is just the canonical simple structure of the topological manifold $E$. Therefore both $o(f)$ and $\tau_{\fib}(f)$ are given by the respective classes of the Whitehead torsion of $\lambda$.
\end{proof}
\end{thm}

%-------------------------------------------------------------
%------------------------------------------------------------

\appendix

\section{Fibering \texorpdfstring{$Q$}{Q}-manifolds}\label{section:fibering_Q_manifolds}

The goal of this appendix is to show that the obstruction theory for both existence and uniqueness developed in this paper also applies to fibering compact $Q$-manifolds over compact ANRs.

Here is a collection of results from the theory of $Q$-manifolds (see \cite{Chapman(1976)}):
\begin{enumerate}
\item If $X$ is a locally compact ANR (e.g.~a topological manifold), then $X\times Q$ is a $Q$-manifold.
\item Every compact $Q$-manifold is of the form $X\times Q$, where $X$ is a compact polyhedron.
\item A map $f\colon X\to Y$ between compact CW complexes is a simple homotopy equivalence if and only if $f\times\id_Q\colon X\times Q\to Y\times Q$ is homotopic to a homeomorphism. This property may be taken as a definition of simple homotopy equivalence between compact ANRs.
\item Any cell-like map \cite{Lacher(1969)} between $Q$-manifolds is arbitrarily close to a homeomorphism. This also shows that $M\times Q\cong Q$.
\end{enumerate}

A compact $Q$-manifold bundle is a fiber bundle whose fibers are compact $Q$-manifolds. Given a fibration $p\colon E\to B$ over a paracompact spaces, the structure space $\calS^Q(p)$ of compact $Q$-manifold bundles is defined in analogy to the structure space $\calS_n(p)$. 

Facts (i) and (iv) show that the total space of a compact $Q$-manifold bundle over a compact ANR is a compact $Q$-manifold. As a consequence there is a geometric assembly map
\[\alpha\colon \calS^Q(p)\to\calS^Q(E)\]
whenever $B$ is a compact ANR. Of course, the re-interpretation of the fibering problem in terms of the geometric assembly map (Lemma \ref{fqm:question_formulated_with_structure_spaces}) remains valid in the $Q$-manifold setting. 

The relation between the fibering problem for compact $Q$-manifolds and the fibering problem of compact topological manifolds is given by the map
\[(\times Q)\colon \calS_n(p)\to\calS^Q(p)\]
which sends a simplex $(q,\lambda)$ to $(q\times Q, \lambda')$ there $\lambda'$ is the obvious composite of $\lambda\times\id_Q$ with the projection $E\times Q\to E$. Since $Q\times I\cong Q$, it factors canonically through $\calS_\infty(p)$.

\begin{thm}
If $B$ is a compact topological manifold and $p\colon E\to B$ is a bundle of compact topological manifolds, then the following diagram is a weak homotopy pull-back:
\[\xymatrix{
{\calS_\infty(p)} \ar[rr]^{(\times Q)} \ar[d]^\alpha && {\calS^Q(p)} \ar[d]^\alpha\\
{\calS_\infty(E)} \ar[rr]^{(\times Q)} && {\calS^Q(E)}
}\]
\end{thm}

\begin{proof}
The diagram commutes up to homotopy. Moreover the vertical tangent bundle defines a map 
\[T^v\colon \calS_\infty(p)\to\map(E,B\TOP)\]
such that the following diagram commutes up to homotopy (see \cite[Proposition 8.3]{Steimle(2011)} for the $T^v$-component):
\[\xymatrix{
{\calS_\infty(p)} \ar[rr]^(.35){((\times Q),T^v)} \ar[d]^\alpha
 && {\calS^Q(p)\times\map(E,B\TOP)} \ar[d]^{\alpha\times(+p^* TB)} \\
{\calS_\infty(E)} \ar[rr]^(.35){((\times Q),T^v)} 
 && {\calS^Q(E)\times\map(E,B\TOP)}
}\]
Since the map $(+p^*TB)$ is an equivalence, we can proof the theorem by showing that the horizontal lines in the diagram are weak homotopy equivalences.

To do that, denote by $\calS_n^{fr}(p)$ the space of framed manifold structures on $p$: It is the geometric realization of a simplicial set where a $k$-simplex is a commutative diagram
\[\xymatrix{
E' \ar[rd]_q \ar[rr]^\varphi_\simeq && E\times \Delta^k \ar[ld]^p\\
& B\times \Delta^k
}\]
together with a bundle map $T^v q\to\epsilon$ from the vertical tangent bundle of $q$ to the trivial bundle $n$-dimensional over $E\times\Delta^k$ which covers $\varphi$. The usual stabilization procedure produces a space $\calS_\infty^{fr}(p)$.

The forgetful map from $\calS_\infty^{fr}(p)$ to $\calS_\infty(p)$ fits into a commutative diagram
\begin{equation}\label{diagram_comparing_structure_spaces}
\xymatrix{
{\calS^{fr}_\infty(p)} \ar[rr] \ar[rrd]_{(\times Q)} 
  && {\calS_\infty(p)} \ar[rr]^{T^v} \ar[d]^{(\times Q)}
  && {\map(E,B\TOP)} \\
  && {\calS^Q(p)}
}
\end{equation}
The proof is now a consequence of the following two claims:

\begin{description}
\item[Claim (i)] The horizontal line in \eqref{diagram_comparing_structure_spaces} is a weak homotopy fibration sequence, and
\item[Claim (ii)] The diagonal arrow in \eqref{diagram_comparing_structure_spaces} is a weak homotopy equivalence.
\end{description}

Claim (i) is Proposition 1.2.1 from \cite{Hoehn(2009)}. To prove claim (ii), we will use the notation and results from \cite[section 2]{Steimle(2011)}. The structure space $\calS_n(p)$ is weakly homotopy equivalent to a space of lifts in the diagram
\[\xymatrix{
 && {\Bun_n(*;F)} \ar[d]\\
B \ar[rr]^p \ar@{.>}[rru] && {\Fib(*;F)}
}\]
where $F$ is the fiber of $p$. Since $B$ is compact, it follows that
\[\calS_\infty(p)\simeq\lift{\Bun_\infty(*;F)}{B}{p}{\Fib(*;F))}\]
with $\Bun_\infty(*;F):=\colim_n \Bun_n(*;F)$, the colimit over the stabilization.

For a space $X$, let $\mathbf{Fib}^{fr}_n(X;F)$ be the category where an object is a fibration $p\colon E\to X$ with fiber $F$ together with a $\TOP(n)$-bundle over $E$, and a morphism is a fiber homotopy equivalence of fibrations which is covered by a bundle map. Let $\mathbf{Bun}^Q(X;F)$ be the category of bundles of compact $Q$-manifolds over $X$ where the fibers are homotopy equivalent to $F$, with bundle homeomorphisms as morphisms. The arguments from \cite[section 2]{Steimle(2011)} produce classifying spaces $\Fib^{fr}_n(*;F)$ and $\Bun^Q(*;F)$. Letting
\[\Fib^{fr}_\infty(*;F):=\colim_n \Fib^{fr}_n(*;F),\]
the colimit over the stabilization of the euclidean bundle, we obtain homotopy equivalences
\[\calS^{fr}_\infty(p)\simeq \lift{\Bun_\infty(*;F)}{B}{(p, \epsilon)}{\Fib^{fr}_\infty(*;F)},\quad \calS^Q(p)\simeq \lift{\Bun^Q(*;F)}{B}{p}{\Fib(*;F)}\]
where $F$ denotes the fiber of $p$.

Hence, claim (ii) follows from

\begin{description}
\item[Claim (ii')] The diagram
\[\xymatrix{
{\Bun_\infty(*;F)} \ar[rr]^{(\times Q)} \ar[d] && {\Bun^Q(*;F)} \ar[d]\\
{\Fib^{fr}_\infty(*;F)} \ar[rr] && {\Fib(*;F)}
}\]
is homotopy cartesian.
\end{description}

We will show claim (ii') by considering the horizontal homotopy fibers. Again by the arguments of \cite[section 2]{Steimle(2011)}, the lower horizontal fiber is given by the mapping space $\map(F,B\TOP)$. 

On the other hand, the upper homotopy fiber over the point determined by a $Q$-manifold $N$ is the stable structure space of compact topological manifolds $M$ equipped with a homeomorphism $M\times Q\to N$. The path components of this space, in view of property (iii) above, are the stable homeomorphism classes of compact topological manifolds equipped with a simple homotopy equivalence to $N$. There are in bijection to $[F,B\TOP]$ by an argument involving the $h$-cobordism theorem (see e.g.~\cite[prop.~5.1]{Wall_IV}). To determine the higher homotopy groups of this structure space, recall that $\Bun_\infty(*;F)$ and $\Bun^Q(*;F)$ are disjoint unions of classifying spaces, and use the homotopy fibration sequence \cite[p.~171]{Weiss-Williams(2001)}
\[\colim_n \TOP(M\times D^n)\to\TOP(M\times Q)\to\map(M,B\TOP)\]
for a compact topological manifold $M$, which is due to Chapman-Ferry and Burghelea \cite{Burghelea(1983)}.
\end{proof}

This theorem shows that for a given map $f\colon M\to B$ between compact topological manifolds, the fibering problem for $f$ and the one for $M\times Q\to B$ are equivalent. Since by result (ii) above, every compact $Q$-manifold is of the form $M\times Q$, this reduces the fibering problem for $Q$-manifolds to the one for topological manifolds provided the base space is a topological manifold.

The case where $B$ is a compact ANR which is not a topological manifold follows by the following ``change of base'' result \cite{Chapman-Ferry(1978)}:

\begin{prop}
Let $f\colon M\to B$ a map from a compact $Q$-manifold to a compact ANR, and let $h\colon B\to B'$ be a simple homotopy equivalence between compact ANRs. Then, $f$ fibers if and only if $h\circ f$ fibers. Moreover, the equivalence classes of $Q$-manifold fiber bundle projections homotopic to $f$ are in bijection to those homotopic to $h\circ f$.
\end{prop}

\begin{proof}
Properties (i) and (iv) above allow a reduction to the case where $B$ is a compact $Q$-manifold. But in this case the claim is obvious since we may assume that $h$ is a homeomorphism.
\end{proof}

%----------------------------------------------
%----------------------------------------------

\end{document}